\documentclass[a4paper,12pt]{article}

\usepackage{amsthm, amsmath, amssymb }
\usepackage[ngerman, english]{babel}
\usepackage[utf8]{inputenc} 
\usepackage[T1]{fontenc} 
\usepackage{enumitem}

\newcommand*{\cB}{\mathcal{B}}

\newcommand*{\cE}{\mathcal{E}}

\newcommand*{\cH}{\mathcal{H}}

\newcommand*{\cL}{\mathcal{L}}

\newcommand*{\cN}{\mathcal{N}}
\newcommand*{\cP}{\mathcal{P}}

\newcommand*{\cS}{\mathcal{S}}
\newcommand*{\cU}{\mathcal{U}}
\newcommand*{\cV}{\mathcal{V}}
\newcommand*{\cW}{\mathcal{W}}

\newcommand{\A}{\mathbb{A}}
\newcommand*{\E}{\mathbb{E}}
\newcommand{\Q}{\mathbb{Q}}
\newcommand{\PP}{\mathbb{P}}

\newcommand*{\C}{\mathbb{C}}
\newcommand*{\N}{\mathbb{N}}
\newcommand*{\Z}{\mathbb{Z}}
\newcommand*{\R}{\mathbb{R}}

\newcommand*{\Rd}{\R^d}

\newcommand{\lb}{\left(}

\newcommand{\lcb}{\left\{}

\newcommand{\lVb}{\left\|}

\newcommand{\lab}{\langle}

\newcommand{\rb}{\right)}

\newcommand{\rcb}{\right\}}

\newcommand{\rVb}{\right\|}

\newcommand{\rab}{\rangle}

\newcommand*{\1}{{\bf 1}}
\newcommand*{\e}{\mathrm{e}}
\newcommand*{\ii}{\mathrm{i}}
\newcommand{\rmd}{\mathrm{d}}
\newcommand{\rmE}{\mathrm{E}}

\newcommand{\ind}{\1_{[0,t)}}
\newcommand{\Hol}{\mathrm{Hol}_0(\cN_\C)}

\newcommand{\wh}{\widehat}
\newcommand{\ot}{\otimes}

\newcommand{\otn}{{\ot n}}

\newcommand{\wot}{\mathbin{\wh{\ot}}}

\newcommand{\wotn}{{\wot n}}

\DeclareMathOperator{\spann}{span}

\newcommand*{\halb}{\frac{1}{2}}
\newcommand{\abs}[1]{\left|#1\right|}
\newcommand{\norm}[1]{\|#1\|}

\newcommand{\ddp}[2]{\lab\!\lab#1,#2\rab\!\rab}

\newtheorem{satz}{Satz}[section]

\newtheoremstyle{tp}
{4,75pt}
{4,75pt}
{\itshape}
{}
{\bf}
{.}
{.5em}
{}
\theoremstyle{tp}

\newtheorem{theorem}[satz]{Theorem}

\newtheoremstyle{mystyle}
{4,75pt}
{4.75pt}
{\normalfont}
{}
{\bf}
{.}
{.5em}
{}
\theoremstyle{mystyle}

\newtheorem{definition}[satz]{Definition}
\newtheorem{proposition}[satz]{Proposition}
\newtheorem{corollary}[satz]{Corollary}
\newtheorem{lemma}[satz]{Lemma}
\newtheorem{remark}[satz]{Remark}
\newtheorem{example}[satz]{Example}

\newenvironment{Beweis*}{{\bf Beweis}}{{\unskip\nobreak\hfil\penalty50\hskip2em\hbox{}\nobreak\hfil\rule{2mm}{2mm}\parfillskip=0pt \finalhyphendemerits=0 \par}\vskip 2ex}

\title{Mittag-Leffler Analysis I: Construction and characterization}
\date{\today}
\author{Martin~Grothaus \and Florian~Jahnert \and Felix~Riemann \and Jos\'e~Lu\'is~da~Silva}

\begin{document}

\maketitle

\begin{abstract}
\noindent We construct an infinite dimensional analysis with respect to non-Gaussian measures of Mittag-Leffler type which we call Mittag-Leffler measures. It turns out that the well-known Wick ordered polynomials in Gaussian analysis cannot be generalized to this non-Gaussian case. Instead of using Wick ordered polynomials we prove that a system of biorthogonal polynomials, called Appell system, is applicable to the Mittag-Leffler measures. Therefore we are able to introduce a test function and a distribution space. As an application we construct Donsker's delta in a non-Gaussian setting as a weak integral in the distribution space. 
\end{abstract}



%
%
%
%


\section{Introduction} 
\noindent During the last decades white noise analysis has evolved into an infinite dimensional distribution theory, with rapid developments in mathematical structure and applications in various domains. For an overview we refer to the monographs \cite{HKPS, Oba, Kuo}. Especially, a deep understanding of the structure of spaces of smooth and generalized random variables over the white noise space or, more generally, Gaussian spaces is provided by various characterization theorems \cite{PS91, KLPSW96, GKS99}. Also, the theory of white noise analysis and the including tools can be applied to a number of fields in mathematics and physics, for example Feynman integration \cite{FPS91, HS83, KS92, LLSW93}, representation of quantum field theory \cite{AHPS89, PS93}, intersection local times for Brownian motion \cite{FHSW97, Wa91} as well as for fractional Brownian motion \cite{DOS08, OSS11}, Dirichlet forms \cite{AHPRS90a, AHPRS90b, HPS88}, infinite dimensional harmonic analysis \cite{Hi89} and so forth.

Almost at the same time, first attempts were made to introduce a non-Gaussian infinite dimensional analysis, by transferring properties of the Gaussian measure to the Poisson measure \cite{Ito88}. This approach can be generalized with the help of a biorthogonal system, which consists of generalized Appell systems \cite{Da91, ADKS96, KSWY98}. It was shown that this approach is suitable for a wide class of measures, including the Gaussian measure and the Poisson measure \cite{KSS97}. Two properties of the measure are essential: An analyticity condition of its Laplace transform and a non-degeneracy or positivity condition (see also \cite{KK99}). In this concept, similar notions and characterizations as in Gaussian analysis can be introduced \cite{KSWY98}.

In this work we concentrate on those measures whose characteristic functions are given via Mittag-Leffler functions. We refer to these measures as Mittag-Leffler measures. The grey noise measure \cite{Sch92, MM09} is included as a special case in the class of Mittag-Leffler measures, which offers the possibility to apply the Mittag-Leffler analysis to fractional differential equations, in particular to fractional diffusion equations \cite{Sch90, Sch92, MLP01}, which carry numerous applications in science, like relaxation type differential equations or viscoelasticity. 

In this paper we prove that the Mittag-Leffler measures belong to the class of measures for which Appell systems exist. Hence we are able to introduce spaces of test functions and distributions with respect to Mittag-Leffler measures. Moreover we construct a distribution which can be considered as a generalization of Donsker's delta from Gaussian analysis.

\section{The finite dimensional Mittag-Leffler measure}
\noindent In this section we recall the Mittag-Leffler measures in the finite dimensional Euclidean space $\Rd,d\in\N$. Using the Gram-Schmidt method we compute the first orthogonal polynomials for that measure and show certain key properties of them.

\begin{definition}
For $0<\beta<\infty$ the Mittag-Leffler function is an entire function defined by its power series
\[
	\rmE_\beta (z) := \sum_{n=0}^\infty \frac{z^n}{\Gamma(\beta n+1)},\quad z\in\C.
\]
Here $\Gamma$ denotes the well-known Gamma function which is an extension of the factorial to complex numbers, such that $\Gamma(n+1) = n!$ for $n\in\N$.
\end{definition}

The Mittag-Leffler function was introduced by G\"osta Mittag-Leffler in \cite{ML05}, see also \cite{W05a, W05b}. In \cite{P48, Fel71} it was shown that for $\beta\in (0,1]$ the mapping $\lcb t\in\R\mid t>0\rcb\ni t\mapsto\rmE_\beta(-t)\in\R$ is completely monotonic. This is sufficient to prove the following lemma, see \cite{P48}:

\begin{lemma}\label{mixingmeasure}
For $\beta\in(0,1]$ there exists a unique probability measure $\nu_\beta$ on $(0,\infty)$ such that
\[
	\rmE_\beta(-t)=\int_0^\infty\e^{-ts}\,\rmd\nu_\beta(s) 
\]
for all $t\ge 0$. For $\beta\in(0,1)$ the measure $\nu_\beta$ is absolutely continuous with respect to the Lebesgue measure on $(0,\infty)$ and the density $g_\beta$ of $\nu_\beta$ is given by \cite{P48, Sch92}:
\[
	g_\beta(t) = \frac{1}{\beta} t^{-1-1/\beta} f_\beta(t^{-1/\beta}), \quad t\ge 0,
\]
where $f_\beta$ denotes the one-sided $\beta$-stable probability density which can be characterized by its Laplace transform
\[
	\int_0^\infty \e^{-tx} f_\beta(x) \,\rmd x = \e^{-t^\beta}, \quad t\ge 0.
\]
\end{lemma}
The density $g_\beta$ can also be expressed in terms of the H-function \cite{Sch86,Sch92}:
\[
	g_\beta(t) = H^{1\,0}_{1\,1}\lb t \left| \begin{matrix} (1-\beta,\beta) \\ (0,1) \end{matrix} \right. \rb, \quad t\ge 0.
\]
This shows that the probability measure $\nu_\beta$ has the $M$-Wright function $M_\beta$ as density, see Lemma \ref{Mbeta=Hfunction} below.

With this result one can show, see \cite{Sch92}, that the mapping 
\[
	\R^d\ni p \mapsto \rmE_\beta\lb-\tfrac{1}{2}(p,p)\rb\in\R
\]
is a characteristic function and thus defines a unique probability measure on $\Rd$:

\begin{definition}
For $\beta\in(0,1]$ we define $\mu_\beta^d$ to be the unique probability measure on $\R^d$ fulfilling
\[ \int_{\R^d}\exp(\ii(p,x))\,\rmd\mu_\beta^d(x)=\rmE_\beta\Bigl(-\halb(p,p)\Bigr) \]
for all $p\in\R^d$.
\end{definition}

In \cite{Sch92} the moments of $\mu_\beta^d$ are computed:

\begin{lemma}\label{lem:moments}
The measure $\mu_\beta^d$ has moments of all order. More precisely, if
\[ M_d^\beta(n_1,\dotsc,n_d):=\int_{\R^d}x_1^{n_1}\dotsm x_d^{n_d}\,\rmd\mu_\beta^d(x) \]
for $n_1,\dotsc,n_d\in\N$, then $M_d^\beta(n_1,\dotsc,n_d)=0$ if any of the $n_i$ is odd, and for the even moments we have
\[ M_d^\beta(2n_1,\dotsc,2n_d)=\frac{(2n_1)!\dotsm(2n_d)!(n_1+\dotsb+n_d)!}{2^{n_1+\dotsb+n_d}n_1!\dotsm n_d!\Gamma(\beta(n_1+\dotsb+n_d)+1)}. \]
\end{lemma}

We apply Gram-Schmidt orthogonalization to the monomials $x^n$, $n\in\N$, to obtain monic polynomials $H_n^\beta$, $n\in\N$, with $\deg H_n^\beta=n$, which are orthogonal in $L^2(\R,\mu_\beta^1)$. These polynomials are determined by the moments of the measure $\mu_\beta^1$. The first five of these polynomials are given by
\begin{align*}
H_0^\beta(x) & = 1, & H_1^\beta(x) & = x,\\ 
H_2^\beta(x) & = x^2-\frac{1}{\Gamma(\beta+1)}, & H_3^\beta(x) & = x^3-\frac{6\Gamma(\beta+1)}{\Gamma(2\beta+1)}\,x,
\end{align*} 
\[
	H_4^\beta(x) = x^4 - c(\beta)x^2 + \frac{c(\beta)}{\Gamma(\beta+1)}-\frac{6}{\Gamma(2\beta+1)},
\]
where
\[ c(\beta)=\frac{90\Gamma(\beta+1)^2\Gamma(2\beta+1)-6\Gamma(\beta+1)\Gamma(3\beta+1)}{6\Gamma(\beta+1)^2\Gamma(3\beta+1)-\Gamma(2\beta+1)\Gamma(3\beta+1)}. \]
Note that for $\beta=1$ the measure $\mu_\beta^d$ is the standard Gaussian measure, hence in this case these polynomials reduce to the Hermite polynomials, which are orthogonal with respect to the weighting function $\exp(-\tfrac{1}{2}x^2)$.

\begin{corollary}\label{cor:notorthogonal}
For $\beta\in(0,1]$ it holds
\[ \int_{\R^2}H_4^\beta(x)H_2^\beta(y)\,\rmd\mu_\beta^2(x,y)=0 \]
if and only if $\beta=1$.
\end{corollary}

\begin{proof}
By Lemma \ref{lem:moments} we have
\begin{align*}
&\int_{\R^2}H_4^\beta(x)H_2^\beta(y)\,\rmd\mu_\beta^2(x,y)\\
 & = \frac{18}{\Gamma(3\beta+1)}-\frac{6}{\Gamma(2\beta+1)\Gamma(\beta+1)}-\frac{2c(\beta)}{\Gamma(2\beta+1)}+\frac{c(\beta)}{\Gamma(\beta+1)^2}\\
 & = A(\beta)
\cdot\Bigl(3\Gamma(2\beta+1)^2-3\Gamma(\beta+1)^2\Gamma(2\beta+1)-\Gamma(\beta+1)\Gamma(3\beta+1)\Bigr),
\end{align*}
where $A(\beta) = 24/\lb \Gamma(2\beta+1)\Gamma(3\beta+1)(6\Gamma(\beta+1)^2-\Gamma(2\beta+1)) \rb\neq 0$. The term inside the bracket is zero if and only if $\beta=1$.
\end{proof}

\begin{remark}
This result can also be used to show that $\mu_\beta^{k+l}$ is the product measure of $\mu_\beta^k$ and $\mu_\beta^l$ for $k,l\ge 1$ if and only if $\beta=1$.
\end{remark}

\section{The Mittag-Leffler measure}\label{Mittag-Leffler measure}
\noindent In this section we repeat the construction of the Mittag-Leffler measure as a probability measure on a conuclear space $\cN'$ from \cite{Sch92}. First we need to collect some facts about nuclear triples used in this paper. For details see e.g.~\cite{Sch71, RS}.

Let $\cH$ be a real separable Hilbert space with inner product $(\cdot,\cdot)$ and corresponding norm $\abs{\cdot}$. Let $\cN$ be a nuclear space which is continuously and densely embedded in $\cH$ and let $\cN'$ be its dual space. The canonical dual pairing between $\cN'$ and $\cN$ is denoted by $\lab\cdot,\cdot\rab$, and by identifying $\cH$ with its dual space via the Riesz isomorphism we get the inclusions $\cN\subset\cH\subset\cN'$. In particular $\lab f,\varphi\rab=(f,\varphi)$ for $f\in\cH$, $\varphi\in\cN$.

We assume that $\cN$ can be represented by a countable family of Hilbert spaces as follows: For each $p\in\N$ let $\cH_p$ be a real separable Hilbert space with norm $\abs{\cdot}_p$ such that $\cN\subset\cH_{p+1}\subset\cH_p\subset\cH$ continuously and the inclusion $\cH_{p+1}\subset\cH_p$ is a Hilbert Schmidt operator. It is no loss of generality to assume $\abs{\cdot}_p\le\abs{\cdot}_{p+1}$ on $\cH_{p+1}$ and $\cH_0=\cH$, $\abs{\cdot}_0=\abs{\cdot}$. The space $\cN$ is assumed to be the projective limit of the spaces $(\cH_p)_{p\in\N}$, that is $\cN=\bigcap_{p\in\N}\cH_p$ and the topology on $\cN$ is the coarsest topology such that all inclusions $\cN\subset\cH_p$ are continuous.

This also gives a representation of $\cN'$ in terms of an inductive limit: Let $\cH_{-p}$ be the dual space of $\cH_p$ with respect to $\cH$ and let the dual pairing between $\cH_{-p}$ and $\cH_p$ be denoted by $\lab\cdot,\cdot\rab$ as well. Then $\cH_{-p}$ is a Hilbert space and we denote its norm by $\abs{\cdot}_{-p}$. It follows by general duality theory that $\cN'=\bigcup_{p\in\N}\cH_{-p}$, and we may equip $\cN'$ with the inductive topology, that is the finest locally convex topology such that all inclusions $\cH_{-p}\subset\cN'$ are continuous.

We end up with the following chain of dense and continuous inclusions:
\[ \cN\subset\cH_{p+1}\subset\cH_p\subset\cH\subset\cH_{-p}\subset\cH_{-(p+1)}\subset\cN'. \]

The Hilbert space tensor product of the space $\cH_p$ is denoted by $\cH_p^\otn$ and the notation $\abs{\cdot}_p$ is kept for the norm on this space. The subspace of symmetric elements is denoted by $\cH_p^\wotn$. The same notations are used for the space $\cH_{-p}$. Then $\cH_{-p}^\otn$ is the dual space of $\cH_p^\otn$ with respect to $\cH^\otn$ and we again use $\lab\cdot,\cdot\rab$ for the dual pairing between these spaces. A simple way to define the tensor powers of $\cN$ is $\cN^\otn:=\bigcap_{p\in\N}\cH_p^\otn$ and equip this space with the projective topology. Then $(\cN^\otn)'=\bigcup_{p\in\N}\cH_{-p}^\otn$ can be equipped with the inductive topology.

To all the real spaces above we also consider their complexifications which will be distinguished by a subscript $\C$, e.g. the complexification of $\cH_p$ is $\cH_{p,\C}$ and so on. In the following we always identify $f = [f_1,f_2]\in\cH_{p,\C}, f_1,f_2\in\cH_p$ for $p\in\Z$ with $f=f_1 + if_2$. 

Similar as in the last section we have that the mapping 
\[ \cN\ni\varphi\mapsto\rmE_\beta\lb-\halb\lab\varphi,\varphi\rab\rb\in\R \]
is a characteristic function on $\cN$. Using the theorem of Bochner and Minlos \cite{BK95}, the following probability measure on $\cN'$, equipped with its cylindrical $\sigma$-algebra, can be defined:

\begin{definition}
For $\beta\in(0,1]$ the Mittag-Leffler measure is defined to be the unique probability measure $\mu_\beta$ on $\cN'$ such that
\[
	\int_{\cN'} \exp(\ii\lab\omega,\varphi\rab)\,\rmd\mu_\beta(\omega) = \rmE_\beta\lb-\tfrac{1}{2}\lab\varphi,\varphi\rab\rb
\]
for all $\varphi\in\cN$. The corresponding $L^p$ spaces of complex-valued functions are denoted by $L^p(\mu_\beta):=L^p(\cN',\mu_\beta;\C)$ for $p\geq 1$ with corresponding norms $\lVb\cdot\rVb_{L^p(\mu_\beta)}$.
\end{definition}

\begin{remark}\label{specialcases}
Several known measures are included in the class of Mittag-Leffler measures:
\begin{itemize}
	\item In the case $\beta=1$ we get the well-known standard Gaussian measure on $\cN'$.
	\item Choosing a concrete realization of the nuclear triple $\cN\subset\cH\subset\cN'$ as in \cite{MM09} the Mittag-Leffler measure becomes the grey noise measure. Moreover, a grey Brownian motion (introduced in \cite{Sch92, KS93} and also in \cite{MM09}) is given by the stochastic process
	\[
		[0,T]\ni t\mapsto B_t^\beta := \lab\cdot,\ind\rab.
	\]
\end{itemize} 
\end{remark}

By comparison of characteristic functions the following is obvious:

\begin{lemma}\label{imagemeasure}
Let $\varphi_1,\dots,\varphi_d\in\cN$ be orthonormal in $\cH$. Then the image measure of $\mu_\beta$ under the mapping $\cN'\ni\omega\mapsto\lb \lab\omega,\varphi_1\rab,\dots,\lab\omega,\varphi_d\rab\rb\in\Rd$ is the finite dimensional Mittag-Leffler measure $\mu_\beta^d$.
\end{lemma}

This directly implies

\begin{corollary}\label{moments2}
Let $\varphi\in\cN$ and $n\in\N$. Then
\[ \int_{\cN'} \lab\omega,\varphi\rab^{2n+1}\,\rmd\mu_\beta(\omega)=0 \]
and
\[ \int_{\cN'} \lab\omega,\varphi\rab^{2n}\,\rmd\mu_\beta(\omega)= \frac{(2n)!}{\Gamma(\beta n+1)2^n}\lab\varphi,\varphi\rab^n. \]
In particular $\lVb\lab\cdot,\varphi\rab\rVb_{L^2(\mu_\beta)}^2=\frac{1}{\Gamma(\beta+1)}\abs{\varphi}^2$.
\end{corollary}

\begin{remark}
Using Corollary \ref{moments2} it is possible to define $\lab\cdot,\eta\rab$ for $\eta\in\cH$ as the $L^2(\mu_\beta)$-limit of $\lab\cdot,\varphi_n\rab$, where $\lb\varphi_n\rb_{n\in\N}$ is a sequence in $\cN$ converging to $\eta$ in $\cH$.
\end{remark}

\begin{definition}
The space of smooth polynomials $\cP(\cN')$ is the space consisting of finite linear combinations of functions of the form $\lab\cdot,\xi\rab^n$, where $\xi\in\cN_\C$ and $n\in\N$. Every smooth polynomial $\varphi$ has a representation
\[ \varphi(\omega)=\sum_{n=0}^N\lab\omega^\otn,\varphi^{(n)}\rab, \]
where $N\in\N$ and $\varphi^{(n)}$ is a finite sum of elements of the form $\xi^\otn$, $\xi\in\cN_\C$.
\end{definition}

In Gaussian analysis, the Wick-ordered polynomials are a system of polynomials $I_n(\xi)\in\cP(\cN')$, $n\in\N$, such that $I_n(\xi)$ is a monic polynomial in $\lab\cdot,\xi\rab$ of degree $n$ with the property that $I_n(\xi)\perp I_m(\zeta)$ whenever $n\neq m$ or $\xi\perp\zeta$ in $\cH$. These polynomials give a useful orthogonal decomposition of the space $L^2(\mu_1)$ and allow to introduce spaces of test functions and distributions easily, see e.g. \cite{Kuo,HKPS}. Unfortunately, no analogon to these Wick polynomials exists for the measure $\mu_\beta$, $\beta\not=1$, which we show below in Theorem \ref{nowickpoly}.

\begin{lemma}\label{Lemma:uniqueness}
For every $\xi\in\cN$ there exists a unique system $I_n(\xi)\in\cP(\cN')$, $n\in\N$, such that $I_n(\xi)=p_{n,\xi}(\lab\cdot,\xi\rab)$ for some monic polynomial $p_{n,\xi}$ on $\R$ or $\C$ of degree $n$ with the property $I_n(\xi)\perp I_m(\xi)$ in $L^2(\mu_\beta)$ for $n\not=m$.
\end{lemma}

\begin{proof}
By Lemma \ref{imagemeasure} it is obvious that the system
\[ I_n(\xi):=\begin{cases}
1 & \text{if }n=0,\\
0 & \text{if }n\not=0,\,\xi=0,\\
\abs{\xi}^n H_n^\beta\bigl(\frac{\lab\cdot,\xi\rab}{\abs{\xi}}\bigr) & \text{if }n\not=0,\,\xi\not=0,
\end{cases} \]
fulfills the desired properties. When $\xi=0$, uniqueness is clear. To see uniqueness for $\xi\neq 0$, note that the function being constantly one is the only monic polynomial of degree $0$. To finish the induction, assume that $I_0(\xi),\dotsc,I_n(\xi)$ are uniquely determined. The affine space of monic polynomials in $\lab\cdot,\xi\rab$ of degree $n+1$ is given by
\[ \lab\cdot,\xi\rab^{n+1}+\spann\{I_0(\xi),\dotsc,I_n(\xi)\}, \]
hence the linear independent conditions $I_{n+1}(\xi)\perp I_k(\xi)$, $k=0,\dotsc,n$, uniquely determine an element from that space.
\end{proof}

\begin{theorem}\label{nowickpoly}
Let $\beta\in(0,1)$ be given and assume $\dim\cH\ge 2$. Then, for the system $I_n(\xi)$ as in Lemma \ref{Lemma:uniqueness}, there exist $n\not=m$ and $\xi,\zeta\in\cN$ such that $I_n(\xi)\not\perp I_m(\zeta)$ in $L^2(\mu_\beta)$.
\end{theorem}

\begin{proof}
Let $\xi,\zeta\in\cN$ be orthonormal in $\cH$. Then
\[ \int_{\cN'}I_4(\xi)\overline{I_2(\zeta)}\,\rmd\mu_\beta=\int_{\R^2}H_4^\beta(x)H_2^\beta(y)\,\rmd\mu_\beta^2(x,y)\neq 0, \]
by Corollary \ref{cor:notorthogonal}.
\end{proof}

\section{Appell Systems}
\noindent Wick-ordered polynomials play an important role in Gaussian analysis since they form an orthonormal system with respect to the $L^2$-scalar product. Thereby one obtains the Wiener-It\^o-Chaos decomposition. In non-Gaussian analysis one uses the so-called Appell systems, which are biorthogonal systems of polynomials. Such systems were first proposed by Daletskii in \cite{Da91} for probability measures with smooth logarithmic derivative. More details and results, including characterization theorems, were obtained in \cite{ADKS96}. In the following we give the construction from \cite{KSWY98}, where the conditions on the measure $\mu$ could be weakened. We also refer to \cite{KK99} for further results. Again we work on the same nuclear triple 
\[
 \cN\subset\cH\subset\cN'
\]
as described in the beginning of Section \ref{Mittag-Leffler measure} and we consider a measure $\mu$ defined on the cylinder $\sigma$-algebra of $\cN'$. This measure shall fulfill the following two properties:
\begin{enumerate}
	\item[(A1)] The measure $\mu$ has an analytic Laplace transform in a neighborhood of zero, i.e. the mapping
	\[
		\cN_\C\ni\varphi\mapsto l_\mu(\varphi) = \int_{\cN'} \exp(\lab\omega,\varphi\rab) \, \rmd\mu(\omega) \in\C
	\]
	is holomorphic in a neighborhood $\cU\subset\cN_\C$ of zero.
	\item[(A2)] For any nonempty open subset $\cU\subset\cN'$ it should hold that $\mu(\cU)> 0$.

\end{enumerate}
Instead of (A2) \cite{KSWY98} require $\mu$ to be non-degenerate, which is a weaker assumption. But in \cite{KK99} it is shown that this assumption is not sufficient to guarantee the embedding of test functions into the space of square integrable function.

In the sequel we show that the Mittag-Leffler measures $\mu_\beta$, $0<\beta<1$, satisfy (A1) and (A2).

\begin{lemma}\label{expint}
For $\varphi\in\cN$ and $\lambda\in\R$ the exponential function $\cN'\ni\omega\mapsto\e^{\abs{\lambda\lab\omega,\varphi\rab}}$ is integrable and
\[
	\int_{\cN'} \e^{\lambda\lab\omega,\varphi\rab}\,\rmd\mu_\beta(\omega) = \rmE_\beta\lb\halb \lambda^2\lab\varphi,\varphi\rab\rb .
\]
\end{lemma}

\begin{proof}
Obviously, the mapping is measurable.
Define the monotone increasing sequence $g_N = \sum_{n=0}^N \frac{1}{n!} \abs{\lab\cdot,\lambda\varphi\rab}^n$. Then we have by Corollary \ref{moments2} that $g_N$ is integrable. Splitting $g_N$ in 
\[
	g_N	= \sum_{n=0}^{\lfloor N/2\rfloor} \frac{1}{(2n)!} \abs{\lab\cdot,\lambda\varphi\rab}^{2n} + \sum_{n=0}^{\lceil N/2\rceil-1} \frac{1}{(2n+1)!} \abs{\lab\cdot,\lambda\varphi\rab}^{2n+1}
\]
we find that the even part converges to $\rmE_\beta(\tfrac{\lambda^2}{2}\lab\varphi,\varphi\rab)$ by Corollary \ref{moments2}. For the odd part, we use the Cauchy-Schwarz inequality and $ab\leq\tfrac{1}{2}(a^2+b^2)$ and again obtain two sums which converge to the Mittag-Leffler function. Thus integrability follow by monotone convergence and
\[
	\int_{\cN'} \e^{\abs{\lambda\lab\omega,\varphi\rab}}\,\rmd\mu_\beta(\omega) = \lim_{N\to\infty} \int_{\cN'} g_N(\omega)\,\rmd\mu_\beta(\omega) < \infty.
\]
Moreover we have by dominated convergence and Corollary \ref{moments2}
\begin{align*}
	\int_{\cN'} \e^{\lambda\lab\omega,\varphi\rab}\,\rmd\mu_\beta(\omega) = \sum_{n=0}^\infty \frac{\lambda^{2n}}{2^n\Gamma(\beta n+1)}\lab\varphi,\varphi\rab^n =\rmE_\beta\lb\halb\lambda^2\lab\varphi,\varphi\rab\rb.
\end{align*}
\end{proof}

Now we show that the Mittag-Leffler measure fulfills (A1):
\begin{theorem}
The mapping
\[
	\cN_\C\ni\varphi\mapsto l_{\mu_\beta}(\varphi)=\int_{\cN'} \e^{\lab\omega,\varphi\rab}\,\rmd\mu_\beta(\omega)\in\C
\]
is a holomorphic mapping from $\cN_\C$ to $\C$.
\end{theorem}

\begin{proof}
Lemma \ref{expint} shows that $l_{\mu_\beta}$ is locally bounded. In fact we get for $\varphi\in\cN_\C$, where $\varphi=\xi_1+i\xi_2$ with $\xi_1,\xi_2\in\cN$
\begin{align*}
	\abs{l_{\mu_\beta}(\varphi)} \leq \int_{\cN'} \e^{\lab\omega,\xi_1\rab}\,\rmd\mu_\beta(\omega) = \rmE_\beta\lb\halb\lab\xi_1,\xi_1\rab\rb <\infty.
\end{align*}
Consider now the mapping $\C\ni z\mapsto f(z):=l_{\mu_\beta}(\varphi_0+z\varphi)$, $\varphi_0,\varphi\in\cN_\C$. We show that $f$ is analytic in $z$. First we remark that $f$ is continuous. To see this, choose $z\in\C$ and a sequence $\lb z_n\rb_{n\in\N}$ in $\C$ with $z_n\to z$ as $n\to\infty$. Then it holds that for sufficiently large $n\in\N$:
\[
	\abs{\exp(\lab\omega,\varphi_0+z_n\varphi\rab)} \leq \exp(\abs{\lab\omega,\varphi_0\rab})\exp\lb (\abs{z}+1)\abs{\lab\omega,\varphi\rab}\rb.
\]
Now continuity follows by Lebesgue's dominated convergence.

Next let $\gamma$ be a closed and bounded curve in $\C$. Since $\gamma$ is compact we can use Fubini and get
\[
	\int_\gamma \int_{\cN'} \exp(\lab\omega,\varphi_0+z\varphi\rab)\,\rmd\mu_\beta(\omega)\,\rmd z = \int_{\cN'} \int_\gamma \exp(\lab\omega,\varphi_0+z\varphi\rab) \,\rmd z\,\rmd\mu_\beta(\omega) = 0,
\]
because the exponential function is holomorphic. Hence we have by Morera's theorem that $f$ is holomorphic in $\C$. This shows that $l_{\mu_\beta}$ is G-holomorphic, which implies holomorphy, see \cite{Di81}. 
\end{proof}

Using the identity principle from complex analysis, we obtain the following corollary:
\begin{corollary}
For each $z\in\C$ and $\theta\in\cN_\C$ we have that
\[
	\int_{\cN'} \exp(z\lab\omega,\theta\rab)\,\rmd\mu_\beta(\omega) = \rmE_\beta\lb\halb z^2\lab\theta,\theta\rab\rb.
\]
\end{corollary}

Now we show that $\mu_\beta$ satisfies (A2). In fact, we prove a statement which is a little bit stronger. For this, we need a refined result of Lemma \ref{mixingmeasure} which shows that the Mittag-Leffler measure is an elliptically contoured measure, see \cite{Mis84}. We give a precise proof of this statement using Lemma 452A from \cite{F06}.

\begin{theorem}
Let $\nu_\beta$ be the probability measure on $(0,\infty)$ as in Lemma \ref{mixingmeasure}. Then
\begin{equation}\label{eq:mixing}
	\mu_\beta=\int_0^\infty\mu^{(s)}\,\rmd\nu_\beta(s), 
\end{equation}
where $\mu^{(s)}$ denotes the centered Gaussian measure on $\cN'$ with variance $s\ge 0$, i.e.
\[ \int_{\cN'}\exp(\ii\lab\omega,\xi\rab)\,\rmd\mu^{(s)}(\omega)=\exp\Bigl(-\frac{s}{2}\lab\xi,\xi\rab\Bigr),\quad\xi\in\cN. \]
\end{theorem}

\begin{proof}
First we need to show that the right-hand side of \eqref{eq:mixing} is indeed a probability measure. Let $\cE$ be the system of cylinder sets in $\cN'$. It generates the $\sigma$-algebra on $\cN'$ and is stable with respect to finite intersections. According to Lemma 452A in \cite{F06} it suffices to show the measurability of
\[ (0,\infty)\ni s\mapsto\mu^{(s)}(E)\in[0,1] \]
for all $E\in\cE$. By definition for any cylinder set $E\in\cE$ there exists $d\in\N$, $\xi_1,\dotsc,\xi_d\in\cN$ and measurable $C_1,\dotsc,C_d\subset\R$, such that
\[ E=\{\omega\in\cN':\lab\omega,\xi_k\rab\in C_k, k=1,\dotsc,d\}. \]
Without loss of generality, we can assume that $\xi_1,\dots,\xi_d$ are linearly independent. 
Let now $M:=(\lab\xi_k,\xi_l\rab)_{k,l=1,\dotsc,d}$ be the Gram-Matrix of $\xi_1,\dotsc,\xi_d$, which is positive definite due to linear independence. Then, by passing to the image measure,
\begin{align*}
\mu^{(s)}(E)
 & = \int_{\cN'} \1_{\lab\cdot,\xi_1\rab\in C_1} \cdot\ldots\cdot \1_{\lab\cdot,\xi_d\rab\in C_d} \,\rmd\mu^{(s)}\\
 & = \frac{1}{\sqrt{(2\pi)^d\det (sM)}}\int_{C_1\times\dotsb\times C_d}\exp\Bigl(-\frac{1}{2s}(x,M^{-1}x)_{\R^d}\Bigr)\,\rmd x
\end{align*}
for all $s\in(0,\infty)$. By continuity of the integrand and Lebesgue's dominated convergence theorem this is a continuous and hence measurable function in $s\in(0,\infty)$. Hence $\nu:=\int_0^\infty\mu^{(s)}\,\rmd\nu_\beta(s)$ is a probability measure on $\cN'$ by Lemma 452A in \cite{F06}. It follows
\begin{align*}
\int_{\cN'}\exp(\ii\lab\omega,\xi\rab)\,\rmd\nu(\omega)
 & = \int_0^\infty\int_{\cN'}\exp(\ii\lab\omega,\xi\rab)\,\rmd\mu^{(s)}(\omega)\,\rmd\nu_\beta(s)\\
 & = \int_0^\infty\exp\Bigl(-\frac{s}{2}\lab\xi,\xi\rab\Bigr)\,\rmd\nu_\beta(s)\\
 & = \rmE_\beta\Bigl(-\halb\lab\xi,\xi\rab\Bigr)
\end{align*}
for all $\xi\in\cN$ and this implies $\mu_\beta=\nu$.
\end{proof}

\begin{theorem}
$\mu_\beta$ satisfies assumption (A2) for all $0<\beta<1$.
\end{theorem}

\begin{proof}
Let $U\subset\cN'$ be open and non-empty. It is well-known that for the Gaussian measures $\mu^{(s)}$ it holds $\mu^{(s)}(U)>0$ for each $s\in(0,\infty)$, hence
\[ \mu_\beta(U)=\int_0^\infty\mu^{(s)}(U)\,\rmd\nu_\beta(s) \]
is strictly positive.
\end{proof}

\subsection{Appell polynomials}
\noindent We want to give a representation for polynomials $\varphi\in\cP(\cN')$ in terms of Appell polynomials. As in \cite{KSWY98} we first introduce the $\mu_\beta$-exponential by
\[
	e_{\mu_\beta}(\varphi;z) = \frac{\e^{\lab z,\varphi\rab}}{l_{\mu_\beta}(\varphi)}, \quad  z\in\cN'_\C, \varphi\in\cN_\C.
\]
This expression is well defined if and only if $l_{\mu_\beta}(\varphi)\neq 0$. Note that $l_{\mu_\beta}(0)=1$ and $l_{\mu_\beta}$ is holomorphic. This implies that there exists a neighborhood $\cU_0\subset\cN_\C$ of zero, such that $l_{\mu_\beta}(\varphi)\neq 0$ for all $\varphi\in\cU_0$.

For $\varphi\in\cU_0$ the ${\mu_\beta}$-exponential can be expanded in a power series, i.e.
\[
	e_{\mu_\beta}(\varphi;z)=\sum_{n=0}^\infty \frac{1}{n!} \lab P^{\mu_\beta}_n(z),\varphi^{(n)} \rab,\quad  z\in\cN'_\C,\varphi\in\cU_0,
\] 
for suitable $P_n^{\mu_\beta}(z)\in\lb\cN^{\hat{\otimes}n}_\C\rb'$, see \cite{KSWY98}. By
\[
	\PP^{\mu_\beta} = \lcb \lab P^{\mu_\beta}_n(\cdot),\varphi^{(n)} \rab : \varphi^{(n)}\in\cN^{\hat{\otimes}n}_\C,\;n\in\N\rcb
\]
we denote the so-called $\PP^{\mu_\beta}$-system. It is shown in \cite{KSWY98} that $\PP^{\mu_\beta}$ coincides with $\cP(\cN')$.

\begin{remark}
The $\mu_\beta$-exponential is given by \[e_{\mu_\beta}(\varphi;x) = \frac{\e^{\lab x,\varphi\rab}}{\rmE_\beta\lb\halb\lab\varphi,\varphi\rab\rb} \] and the first Appell polynomials can be calculated to be
\begin{align*}
\lab P^{\mu_\beta}_0 (z),\varphi^{\otimes 0}\rab &= 1, \\
\lab P^{\mu_\beta}_1 (z),\varphi^{\otimes 1}\rab &= \lab z,\varphi\rab, \\ 
\lab P^{\mu_\beta}_2 (z),\varphi^{\otimes 2}\rab &= \lab z^{\otimes 2},\varphi^{\otimes 2}\rab - \frac{1}{\Gamma(\beta+1)}\lab\varphi,\varphi\rab, \\
\lab P^{\mu_\beta}_3 (z),\varphi^{\otimes 3}\rab &= \lab z^{\otimes 3},\varphi^{\otimes 3}\rab - \frac{3}{\Gamma(\beta+1)}\lab\varphi,\varphi\rab\lab z,\varphi\rab.
\end{align*} 
Compared to the orthogonal polynomials in Lemma \ref{Lemma:uniqueness}, we see that 
\[
	\lab P^{\mu_\beta}_3(\cdot) , \varphi^{\otimes 3} \rab \neq I_3(\varphi) = \lab\cdot,\varphi\rab^3 - \frac{6\Gamma(\beta+1)}{\Gamma(2\beta+1)}\lab\varphi,\varphi\rab\lab\cdot,\varphi\rab.
\]
Consequently, due to uniqueness of the orthogonal system, the $\PP^{\mu_\beta}$-system is not orthogonal if $\beta\neq 1$. In the case $\beta=1$ the Appell polynomials and the polynomials in Lemma \ref{Lemma:uniqueness} coincide and they are equal to the Hermite polynomials.
\end{remark}

\subsection{The biorthogonal system}
\noindent The space of smooth polynomials $\cP(\cN')$ shall be equipped with the natural topology, such that the mapping
\[
	\varphi = \sum_{n=0}^\infty \lab \cdot^\otn,\varphi^{(n)}\rab \leftrightarrow \vec{\varphi}=\lcb \varphi^{(n)} : n\in\N \rcb
\]
becomes a topological isomorphism from $\cP(\cN')$ to the topological direct sum of tensor powers $\cN^\wotn_\C$, i.e.
\[
	\cP(\cN') \simeq \bigoplus_{n=0}^\infty \cN^\wotn_\C
\]
(note that $\varphi^{(n)}\neq 0$ only for finitely many $n\in\N$). Then we introduce the space $\cP'_{\mu_\beta}(\cN')$ as the dual space of $\cP(\cN')$ with respect to $L^2(\mu_\beta)$, i.e.
\[
	\cP(\cN') \subset L^2(\mu_\beta) \subset \cP'_{\mu_\beta}(\cN')
\]
and the dual pairing $\ddp{\cdot}{\cdot}_{\mu_\beta}$ between $\cP'_{\mu_\beta}(\cN')$ and $\cP(\cN')$ is an extension of the scalar product on $L^2({\mu_\beta})$ by
\[
	\ddp{f}{\varphi}_{\mu_\beta} = (f,\overline{\varphi})_{L^2({\mu_\beta})},\quad\varphi\in\cP(\cN'),f\in L^2({\mu_\beta}).
\] 
Note that (A1) ensures that $\cP(\cN')\subset L^2(\mu_\beta)$ is dense, see \cite{Sko74}.

The aim is now to describe the distributions in $\cP'_{\mu_\beta}(\cN')$ in a similar way as the smooth polynomials, i.e.~we find elements $\Phi^{(n)}\in\lb\cN^{\hat{\otimes}n}_\C\rb'$ and an operator $Q^{\mu_\beta}_n$ on $\lb\cN^{\hat{\otimes}n}_\C\rb'$, such that
\[
	\Phi = \sum_{n=0}^\infty Q^{\mu_\beta}_n\lb\Phi^{(n)}\rb
\]
and moreover a certain biorthogonality relation should hold, see Theorem \ref{biorth} below. To find this required $Q^{\mu_\beta}_n$, we proceed as in \cite{KSWY98} and define a differential operator on $\cP(\cN')$ depending on $\Phi^{(n)}\in\lb\cN^{\hat{\otimes}n}_\C\rb'$ by
\[
	D\lb\Phi^{(n)}\rb\lab\omega^{\otimes m},\varphi^{(m)}\rab := \begin{cases} \frac{m!}{(m-n)!}\lab\omega^{\otimes (m-n)}\hat{\otimes}\Phi^{(n)},\varphi^{(m)}\rab, & m\geq n \\ 0, & m<n \end{cases}
\]
for a monomial $\omega\mapsto\lab\omega^{\otimes m},\varphi^{(m)}\rab$ with $\varphi^{(m)}\in\cN^{\hat{\otimes}m}_\C$. 
If $\cN$ is the space of Schwartz test functions $\cS(\R)$ and $\cH=L^2(\R,dx)$, then for $n=1$ and $\Phi^{(1)}=\delta_t\in\cN_\C'$ this differential operator coincides with the Hida derivative, see \cite{HKPS}.
For each $\Phi^{(n)}\in\lb\cN^{\hat{\otimes}n}_\C\rb'$ the operator $D(\Phi^{(n)})$ is continuous from $\cP(\cN')$ to $\cP(\cN')$, see \cite{KSWY98}, and this enables us to define the dual operator $D(\Phi^{(n)})^* \colon \cP'_{\mu_\beta}(\cN')\to\cP'_{\mu_\beta}(\cN')$. We set $Q^{\mu_\beta}_n(\Phi^{(n)})=D(\Phi^{(n)})^*\1$ for $\Phi^{(n)}\in\lb\cN^{\hat{\otimes}n}_\C\rb'$ and denote the so-called $\Q^{\mu_\beta}$-system in $\cP'_{\mu_\beta}(\cN')$ by
\[
	\Q^{\mu_\beta} = \lcb Q^{\mu_\beta}_n(\Phi^{(n)}) : \Phi^{(n)}\in\lb\cN^{\hat{\otimes}n}_\C\rb',n\in\N\rcb.
\]
As in \cite{KSWY98} we have that for each $\Phi\in\cP'_{\mu_\beta}(\cN')$ there is a unique sequence of kernels $\lb\Phi^{(n)}\rb_{n\in\N} \subset \lb\cN^{\hat{\otimes}n}_\C\rb'$, such that
\begin{equation}\label{eq:chaosdecomposition}
	\Phi=\sum_{n=0}^\infty Q^{\mu_\beta}_n(\Phi^{(n)}).
\end{equation}
Vice versa, every such sum is a generalized function in $\cP'_{\mu_\beta}(\cN')$. The pair $\lb \PP^{\mu_\beta},\Q^{\mu_\beta} \rb$ is called Appell system $\A^{\mu_\beta}$ generated by the measure ${\mu_\beta}$. 

By the next theorem from \cite{KSWY98}, the main achievement of Appell systems becomes apparent, since it proves a biorthogonality relation between $\PP^{\mu_\beta}$ and $\Q^{\mu_\beta}$:

\begin{theorem}\label{biorth}
For $\Phi^{(n)}\in\lb\cN^{\hat{\otimes}n}_\C\rb'$ and $\varphi^{(m)}\in\cN^{\hat{\otimes}m}_\C$ we have
\[
	\ddp{Q^{\mu_\beta}_n\lb\Phi^{(n)}\rb}{\lab P^{\mu_\beta}_m,\varphi^{(m)}\rab}_{\mu_\beta} = \delta_{m,n} n! \lab \Phi^{(n)},\varphi^{(n)}\rab,\quad n,m\in\N_0.
\]
\end{theorem}

\subsection{Test functions and distributions}
\noindent With the help of the Appell system $\A^{\mu_\beta}$ a test function and a distribution space can now be constructed, see \cite{KSWY98}. For $\varphi=\sum_{n=0}^N \lab P^{\mu_\beta}_n,\varphi^{(n)}\rab\in\cP(\cN')$ we define for any $p,q\in\N_0$ the norm
\[
	\norm{\varphi}^2_{p,q,{\mu_\beta}} := \sum_{n=0}^N (n!)^2 2^{nq}\abs{\varphi^{(n)}}_p^2.
\]
By $\lb\cH_p\rb^1_{q,{\mu_\beta}}$ we denote the completion of $\cP(\cN')$ with respect to $\norm{\cdot}_{p,q,{\mu_\beta}}$. Under the condition (A2), \cite{KK99} shows that there are $p',q'\in\N$, such that $(\cH_p)^1_{q,{\mu_\beta}}$ can be topologically embedded in $L^2({\mu_\beta})$ for all $p>p',q>q'$. The test function space $\lb\cN\rb^1_{\mu_\beta}$ is defined as the projective limit of $\bigl(\lb\cH_p\rb^1_{q,{\mu_\beta}}\bigr)_{p,q\in\N}$, i.e.
\[
	\lb\cN\rb^1_{\mu_\beta} := \underset{p,q\in\N}{\operatorname{pr\,lim}}\lb\cH_p\rb^1_{q,{\mu_\beta}}.
\]
As in \cite{KSWY98} this is a nuclear space, which is continuously embedded in $L^2({\mu_\beta})$ and it turns out that the test function space $(\cN)^1_\mu$ is the same for all measures $\mu$ satisfying (A1) and (A2), thus we will just use the notation $(\cN)^1$.

\begin{example}\label{ex:exponentials}
It holds that $\norm{e_{\mu_\beta}(\varphi;\cdot)}^2_{p,q,\mu} = \sum_{n=0}^\infty 2^{nq}\abs{\varphi}^{2n}_p$. Hence we get, that the $\mu$-exponential is not in $(\cN)^1$ if $\varphi\neq 0$, but $e_{\mu_\beta}(\varphi;\cdot)\in (\cH_p)^1_{q,\mu_\beta}$ if $\varphi\in U_{p,q}:=\lcb \varphi\in\cN : 2^q\abs{\varphi}_p<1\rcb$. Moreover, the set 
\[
	\lcb e_\mu(\theta;\cdot) : 2^q\abs{\theta}_p<1,\theta\in\cN_\C\rcb
\] 
is total in $(\cH_p)_{q,\mu_\beta}^1$, see \cite{KSWY98}.
\end{example}

Following \cite{KSWY98} we use the representation $\Phi=\sum_{n=0}^\infty Q^{\mu_\beta}_n(\Phi^{(n)}) \in\cP'_{\mu_\beta}(\cN')$ in order to define the norms
\[
	\norm{\Phi}^2_{-p,-q,{\mu_\beta}} := \sum_{n=0}^\infty 2^{-qn}\abs{\Phi}^2_{-p},\quad p,q\in\N_0.
\]
By $(\cH_{-p})^{-1}_{-q,{\mu_\beta}}$ we denote the set of all $\Phi\in\cP'_{\mu_\beta}(\cN')$ for which $\norm{\Phi}_{-p,-q,{\mu_\beta}}$ is finite. It holds that $(\cH_{-p})^{-1}_{-q,{\mu_\beta}}$ is the dual of $(\cH_p)^1_{q,{\mu_\beta}}$, see \cite{KSWY98}. The space of distributions $(\cN)_{\mu_\beta}^{-1}$ is defined as the inductive limit
\[
	(\cN)_{\mu_\beta}^{-1} = \underset{p,q\in\N}{\operatorname{ind\,lim}}(\cH_{-p})^{-1}_{-q,{\mu_\beta}}.
\] 
As in \cite{KSWY98} $(\cN)_{\mu_\beta}^{-1}$ is the dual of $(\cN)^1$ with respect to $L^2({\mu_\beta})$ and the dual pairing between a distribution $\Phi=\sum_{n=0}^\infty\lab Q^{\mu_\beta}_n\lb\Phi^{(n)}\rb$ with a test function $\varphi=\sum_{n=0}^\infty \lab P^{\mu_\beta}_n,\varphi^{(n)}\rab$ is given by Theorem \ref{biorth} as
	\[
		\ddp{\Phi}{\varphi}_{\mu_\beta} = \sum_{n=0}^\infty n! \lab\Phi^{(n)},\varphi^{(n)}\rab.
	\]
We shall use the same notation for the dual pairing between $(\cH_{-p})^{-1}_{-q,{\mu_\beta}}$ and $(\cH_p)^1_{q,{\mu_\beta}}$.

\subsection{Integral transforms and characterization}
\noindent Since $\mu_\beta$ satisfies (A1), there is $p'_{\mu_\beta}\in\N$ and $\varepsilon_{\mu_\beta}>0$ such that $$\int_{\cN'} \e^{\varepsilon_{\mu_\beta}\abs{\omega}_{-p'_{\mu_\beta}}}\,\rmd{\mu_\beta}(\omega) <\infty,$$ see \cite{KSWY98}. Thus, if $\varphi\in\cV_0=\lcb\varphi\in\cN_\C\mid 2\abs{\varphi}_{p'_{\mu_\beta}}\leq\varepsilon_{\mu_\beta}\rcb$ we have that $\e^{\lab\cdot,\varphi\rab}\in L^2({\mu_\beta})$. We define the Laplace transform of $f\in L^2({\mu_\beta})$ by
\[
	L_{\mu_\beta} f(\varphi) = \int_{\cN'} f(\omega)\e^{\lab\omega,\varphi\rab}\,\rmd{\mu_\beta}(\omega),\quad \varphi\in\cV_0.
\]
Moreover we introduce for all $\varphi\in\cU_0\cap\cV_0$ the $S_{\mu_\beta}$-transform of $f\in L^2(\mu_\beta)$ by
\[
	S_{\mu_\beta} f(\varphi) := \frac{L_{\mu_\beta} f(\varphi)}{l_{\mu_\beta}(\varphi)} = \int_{\cN'} f(\omega) e_{\mu_\beta}(\varphi;\omega)\,\rmd{\mu_\beta}(\omega).
\]
We would like to extend the definition of the $S_{\mu_\beta}$-transform to $(\cN)^{-1}_{\mu_\beta}$. Note that for $\Phi\in (\cN)_{\mu_\beta}^{-1}$ there are $p,q\in\N$, such that $\Phi\in (\cH_{-p})^{-1}_{-q,{\mu_\beta}}$. Moreover for $\varphi\in U_{p,q}$ as in Example \ref{ex:exponentials} we have that $e_{\mu_\beta}(\varphi,\cdot)\in (\cH_p)^1_{q,{\mu_\beta}}$. Thus we can define:
\begin{equation}\label{Def:S-transform}
	S_{\mu_\beta}\Phi(\varphi) = \ddp{\Phi}{e_{\mu_\beta}(\varphi,\cdot)}_{\mu_\beta}.
\end{equation}
Using the representation of $\Phi$ as in \eqref{eq:chaosdecomposition} we have
\begin{equation}\label{stransform}
	S_{\mu_\beta}\Phi(\varphi) = \sum_{n=0}^\infty \lab\Phi^{(n)},\varphi^{\otimes n}\rab,\quad\varphi\in U_{p,q}.
\end{equation}

The space $(\cN)^{-1}_{\mu_\beta}$ can be characterized via the $S_{\mu_\beta}$-transform using spaces of holomorphic functions on $\cN_\C$. By $\Hol$ we denote the space of all holomorphic functions at zero. Let $F$ and $G$ be holomorphic on a neighborhood $\cV,\cU\subset\cN_\C$ of zero, respectively. We identify $F$ and $G$ if there is a neighborhood $\cW\subset\cV$ and $\cW\subset\cU$ such that $F(\xi) = G(\xi)$ for all $\xi\in\cW$. $\Hol$ is the union of the spaces
\[
	\lcb F\in\Hol \biggm| n_{p,l,\infty}(F) = \sup_{\abs{\theta}_p\leq2^{-l}} \abs{F(\theta) } <\infty\rcb,\quad p,l\in\N 
\]
and carries the inductive limit topology.

The following theorem is proved in \cite{KSWY98}:
\begin{theorem}\label{characterization}
The $S_{\mu_\beta}$-transform is a topological isomorphism from $(\cN)^{-1}_{\mu_\beta}$ to $\Hol$. Moreover, if $F\in\Hol,\;F(\theta) = \sum_{n=0}^\infty \lab\Phi^{(n)},\theta^\otn\rab$ for all $\theta\in\cN_\C$ with $\abs{\theta}_p\leq 2^{-l}$ and if $p'>p$ with $\norm{i_{p',p}}_{HS}<\infty$ and $q\in\N$ such that $\rho := 2^{2l-q}\e^2\norm{i_{p',p}}_{HS}^2<1$ then $\Phi = \sum_{n=0}^\infty Q^{\mu_\beta}_n(\Phi^{(n)})\in\lb\cH_{-p'}\rb^{-1}_{-q}$ and
\[
	\norm{\Phi}_{-p',q,{\mu_\beta}} \leq n_{p,l,\infty}(F) (1-\rho)^{-1/2}.
\]
\end{theorem}

As a corollary from the characterization theorem we now present a result which characterizes the integrable mappings with values in $(\cN)^{-1}_{\mu_\beta}$ in a weak sense: 

\begin{theorem}\label{charint}
Let $(T,\cB,\nu)$ be a measure space and $\Phi_t\in (\cN)^{-1}_{\mu_\beta}$ for all $t\in T$. Let $\cU\subset\cN_\C$ be an appropriate neighborhood of zero and $C<\infty$, such that:
\begin{enumerate}[label=\textit{(\roman*)}]
	\item $S_{\mu_\beta}\Phi_\cdot(\xi)\colon T\to\C$ is measurable for all $\xi\in\cU$.
	\item $\int_T \abs{S_{\mu_\beta}\Phi_t(\xi)}\,\rmd\nu(t) \leq C$ for all $\xi\in\cU$.
\end{enumerate}
Then there exists $\Psi\in (\cN)^{-1}_{\mu_\beta}$ such that for all $\xi\in\cU$
\[
	S_{\mu_\beta}\Psi(\xi) = \int_T S_{\mu_\beta}\Phi_t(\xi)\,\rmd\nu(t).
\]
We denote $\Psi$ by $\int_T \Phi_t\,\rmd\nu(t)$ and call it the weak integral of $\Phi$.
\end{theorem}


\begin{remark}
If in addition the mapping
\[
	(\cH_{p'})_{q',\mu_\beta}^1 \supset\spann\lcb e_{\mu_\beta}(\xi,\cdot) \mid \xi\in U_{p',q'} \rcb \ni \varphi \mapsto \ddp{\Phi_\cdot}{\varphi}_{\mu_\beta} \in L^1(T,\nu)
\]
is continuous for large enough $p',q'\in\N$, then $\Psi$ coincides with the Pettis integral.
\end{remark}

\begin{proof}
First we show that the mapping 
\[
	\cU\ni\xi\mapsto F(\xi):=\int_T S_{\mu_\beta}\Phi_t(\xi)\,\rmd\nu(t)
\]
is holomorphic on $\cU$. Let $\xi_0\in\cN_\C$ and $ U\subset\C$ open and small enough, such that $z\xi+\xi_0\in\cU$ for all $\xi\in\cU$ and for all $z\in\overline{U}$. Furthermore let $(z_n)_{n\in\N}$ be a sequence in $U$, $\lim_{n\to\infty} z_n=z\in\overline{U}$. Since $S_{\mu_\beta}\Phi_t$ is holomorphic on $\cU$ we find $z^*\in\overline{U}$ such that for all $n\in\N$
\[
	\abs{S_{\mu_\beta}\Phi_t(\xi_0+z_n\xi)} \leq \abs{S_{\mu_\beta}\Phi_t(\xi_0+z^*\xi)}.
\]
Hence we have by dominated convergence that $U\ni z\mapsto F(\xi_0+z\xi)$ is continuous. Choose now any closed and bounded curve $\gamma$ in $U$. Since $\gamma$ is compact, we have that $\int_\gamma \int_T \abs{S_{\mu_\beta}\Phi_t(\xi_0+z\xi)}\,\rmd\nu(t)\,\rmd z<\infty$ and by Fubini we obtain
\[
	\int_\gamma F(\xi_0+z\xi)\,\rmd z = \int_T \int_\gamma S_{\mu_\beta}\Phi_t(\xi_0+z\xi)\,\rmd z\,\rmd\nu(t)=0,
\]
since the $S_{\mu_\beta}$-transform is an holomorphic function. This shows that $F$ is G-holomorphic. Note that $\abs{F(\xi)} \leq \int_T \abs{S_{\mu_\beta}\Phi_t(\xi)}\,\rmd\nu(t) \leq C$, hence $F$ is also locally bounded and thus holomorphic on $\cU$, see \cite{Di81}. By Theorem \ref{characterization} there exists $\Psi\in (\cN)^{-1}_{\mu_\beta}$ such that
\[
	S_{\mu_\beta}\Psi(\xi) = \int_T S_{\mu_\beta}\Phi_t(\xi)\,\rmd\nu(t)
\]
for all $\xi\in\cU$.

\end{proof}

\section{Donsker's Delta in Mittag-Leffler analysis}
\noindent In Gaussian analysis, Donsker's delta is an important example for a Hida distribution with many applications in quantum field theory, in the theory of stochastic differential equations and in mathematical finance, see \cite{HKPS, W95, AOU01} and references therein. In this section we introduce a distribution in $(\cN)^{-1}_{\mu_\beta}$ being the analog of Donsker's delta in Gaussian analysis. The strategy is as follows:

Use the integral representation for the Dirac delta distribution $\delta$ and give sense to the expression
\[
	\delta(\lab\cdot,\eta\rab) = \frac{1}{2\pi} \int_\R \e^{ix\lab\cdot,\eta\rab}\,\rmd x,\quad\eta\in\cH,
\]
by using Theorem \ref{charint}. We first remark that $\exp(ix\lab\cdot,\eta\rab)\in L^2(\mu_\beta)$ since the absolute value is $\abs{\exp(ix\lab\omega,\eta\rab)}^2=1$. To proceed further, we introduce (as in Gaussian analysis) the following $T_{\mu_\beta}$-transform:

\begin{lemma}
Let $\Phi\in\lb\cN\rb^{-1}_{\mu_\beta}$ and $p,q\in\N$ such that $\Phi\in(\cH_{-p})_{-q,\mu_\beta}^{-1}$.
Then the $T_{\mu_\beta}$-transform given by
\[
	T_{\mu_\beta}\Phi(\varphi) = \ddp{\Phi}{\exp(i\lab\cdot,\varphi\rab}_{\mu_\beta}
\]
is well-defined for $\varphi\in U_{p,q}$ as in Example \ref{ex:exponentials} and we have:
\[
	T_{\mu_\beta}\Phi(\varphi) = \rmE_\beta(-1/2 \lab\varphi,\varphi\rab) S_{\mu_\beta} \Phi(i\varphi).
\]
In particular, $T_{\mu_\beta}\Phi\in\Hol$ if and only if $S_{\mu_\beta}\Phi\in\Hol$ and Theorem \ref{characterization} and Theorem \ref{charint} also hold if the $S_{\mu_\beta}$-transform is replaced by the $T_{\mu_\beta}$-transform.
\end{lemma}

\begin{proof}
For $\Phi\in(\cH_{-p})_{-q,\mu_\beta}^{-1}$ and $\varphi\in U_{p,q}$ the $S_{\mu_\beta}$-transform as in \eqref{Def:S-transform} is well-defined and we have using Lemma \ref{expint}:
\begin{align*}
	S_{\mu_\beta}\Phi(i\varphi) = \ddp{\Phi}{e_{\mu_\beta}(i\varphi,\cdot)}_{\mu_\beta} = \frac{1}{\rmE_\beta (-1/2\lab\varphi,\varphi\rab)}\ddp{\Phi}{\exp(i\lab\cdot,\varphi\rab}_{\mu_\beta}.
\end{align*}
Since the $S_{\mu_\beta}-$ and the $T_{\mu_\beta}$-transform differ only by a factor from $\Hol$ and since $\Hol$ is an algebra, the second assertion follows immediately.
\end{proof}

Now we calculate the $T_{\mu_\beta}$-transform of the square-integrable integrand $\exp(ix\lab\cdot,\eta\rab)$ for $\varphi\in\cN_\C$ with the help of Lemma \ref{expint}:
\begin{align}
	T_{\mu_\beta}\exp(ix\lab\cdot,\eta\rab)(\varphi) &= \frac{1}{2\pi}\int_{\cN'} \exp(i\lab\omega,x\eta+\varphi\rab)\,\rmd\mu_\beta(\omega) \nonumber \\
	&= \frac{1}{2\pi} \rmE_\beta\Bigl(-\frac{1}{2}x^2\lab\eta,\eta\rab -\frac{1}{2} \lab\varphi,\varphi\rab - x\lab\varphi,\eta\rab \Bigr). \label{eq:Ttransform}
\end{align}
In a next step we show that $\int_\R \abs{T_{\mu_\beta}\exp(ix\lab\cdot,\eta\rab)(\varphi)}\,\rmd x$ is bounded for all $\varphi$ in a neighborhood of zero of $\cN_\C$.

\begin{proposition}\label{Prop:Integral}
For $\eta\in\cH$, $\varphi\in\cN_\C$ and $x\in\R$ let $z(x,\eta,\varphi)=\tfrac{1}{2} x^2\lab\eta,\eta\rab + \tfrac{1}{2}\lab\varphi,\varphi\rab + x\lab\eta,\varphi\rab$. Then there exists a constant $C<\infty$ such that
\[
	\int_\R \abs{\rmE_\beta(-z(x,\eta,\varphi))}\,\rmd x \leq C,\quad\varphi\in\cU := \lcb \varphi\in\cN_\C \mid \abs{\varphi} < M \rcb
\]
for any $0<M<\infty$.
\end{proposition}

\begin{proof}
First note that $\rmE_\beta(-z(x,\eta,\varphi))$ is obviously measurable. Now use Corollary \ref{Ebeta=LaplaceMbeta} below and $\varphi = \varphi_1 + i\varphi_2$ for $\varphi_1,\varphi_2\in\cN$ and obtain:
\begin{align*}
	&\int_\R \abs{\rmE_\beta(-z(x,\eta,\varphi))}\,\rmd x	\leq \int_0^\infty M_\beta(r) \int_\R \exp\lb-r\Re(z(x,\eta,\varphi))\rb\,\rmd x\,\rmd r \\
	&=\sqrt{\frac{2\pi}{\lab\eta,\eta\rab}} \int_0^\infty M_\beta(r) r^{-1/2} \exp\lb-r\lb\halb\abs{\varphi_1}^2-\halb\abs{\varphi_2}^2 - \frac{\lab\eta,\varphi_1\rab^2}{2\lab\eta,\eta\rab}\rb \rb \,\rmd r.
\end{align*}
Since $\varphi$ is bounded, we have by Cauchy-Schwarz inequality that
\begin{align*}
	\abs{\varphi_1}^2-\abs{\varphi_2}^2 - \frac{\lab\eta,\varphi_1\rab^2}{\lab\eta,\eta\rab} 
	> \abs{\varphi_1}^2 - M^2 - \frac{\abs{\eta}^2 \abs{\varphi_1}^2}{\lab\eta,\eta\rab} = -M^2
\end{align*}
This yields that
\[
	\int_\R \abs{\rmE_\beta(-z(x,\eta,\varphi))}\,\rmd x \leq \sqrt{\frac{2\pi}{\lab\eta,\eta\rab}} \int_0^\infty M_\beta(r) r^{-1/2} \exp\lb\frac{1}{2}M^2 r \rb \,\rmd r.
\]
But this integral is finite as shown in Lemma \ref{intexists} below.
\end{proof}

\begin{theorem}\label{theorem:Donsker}
Let $0\neq\eta\in\cH$. Then Donsker's delta is defined via the integral
\[
	\delta(\lab\cdot,\eta\rab) := \frac{1}{2\pi}\int_\R \exp\lb ix\lab\cdot,\eta\rab \rb\,\rmd x,
\]
and exists in the space $\lb\cN\rb^{-1}_{\mu_\beta}$ as a weak integral in the sense of Theorem \ref{charint}. Moreover for all $\varphi\in\cU$, $\cU$ as in Proposition \ref{Prop:Integral} we have
\[
	\lb T_{\mu_\beta}\delta(\lab\cdot,\eta\rab)\rb(\varphi) = \frac{1}{\sqrt{2\pi \lab\eta,\eta\rab}} H^{1\,1}_{1\,2}\lb \halb\lab\varphi,\varphi\rab - \frac{\lab\eta,\varphi\rab^2}{2\lab\eta,\eta\rab} \left| \begin{matrix} (1/2,1) \\ (0,1),(1/2\beta,\beta) \end{matrix} \right. \rb,
\]
where $H$ denotes Fox-H-function, see Appendix \ref{AppendixI}.
\end{theorem}

\begin{proof}
Due to Proposition \ref{Prop:Integral} there exists $C<\infty$ such that
\[
	\frac{1}{2\pi} \int_\R \lb T_{\mu_\beta}\exp(ix\lab\cdot,\eta\rab)\rb(\varphi)\,\rmd x < C,\quad\varphi\in\cU.
\]
Thus, by Theorem \ref{charint} the existence of $\delta(\lab\cdot,\eta\rab)\in (\cN)^{-1}_{\mu_\beta}$ follows. Finally we calculate the $T_{\mu_\beta}$-transform of Donsker's delta using Corollary \ref{Ebeta=LaplaceMbeta} below:
\begin{align*}
	&T_{\mu_\beta}\delta(\lab\cdot,\eta\rab)(\varphi) = \frac{1}{2\pi} \int_\R \rmE_\beta(-z(x,\eta,\varphi))\,\rmd x \\
	&= \frac{1}{2\pi} \int_0^\infty M_\beta(r) \exp\lb -\halb r\lab\varphi,\varphi\rab\rb \int_\R \exp\lb -\halb r\lab\eta,\eta\rab x^2 - r\lab\eta,\varphi\rab x\rb\,\rmd x\rmd r \\
	&= \frac{1}{\sqrt{2\pi\lab\eta,\eta\rab}} \int_0^\infty M_\beta(r) r^{-1/2} \exp\lb-r\lb \halb\lab\varphi,\varphi\rab - \frac{\lab\eta,\varphi\rab^2}{2\lab\eta,\eta\rab}\rb\rb\,\rmd r.
\end{align*}
Using Lemma \ref{LaplaceMbeta} below we get the desired result.
\end{proof}

\begin{remark}\label{Rem:seriesexpansion}
With the help of \eqref{eq:Hseries} below we can find the series expansion of the $T_{\mu_\beta}$-transform of Donsker's delta. In fact we have
\[
	H^{1\,1}_{1\,2}\lb z \left| \begin{matrix} (1/2,1) \\ (0,1),(1/2\beta,\beta) \end{matrix} \right. \rb = \sum_{k=0}^\infty \frac{(-1)^k \Gamma(k+1/2)}{k! \Gamma(1+\beta(k-1/2))}z^k,\quad z\in\C\setminus\lcb 0\rcb.
\]
Note that in the case $\beta=1$:
\[
	\sum_{k=0}^\infty \frac{(-1)^k \Gamma(k+1/2)}{k! \Gamma(1+k-1/2)}z^k = \e^{-z}.
\]
Thus our definition of Donsker's delta for $\beta=1$ becomes the usual Donsker's delta as known in Gaussian analysis.
\end{remark}

\begin{corollary}
The generalized expectation of Donsker's delta is given by 
\[
	\E\lb\delta\lb\lab\cdot,\eta\rab\rb\rb = \lb T_{\mu_\beta} \delta\lb\lab\cdot,\eta\rab\rb\rb(0) =  \ddp{\delta(\lab\cdot,\eta\rab)}{1}_{\mu_\beta}.
\]
Using the series expansion from Remark \ref{Rem:seriesexpansion}, we get:
\[
	\E\lb\delta\lb\lab\cdot,\eta\rab\rb\rb = \frac{1}{\sqrt{2\pi\lab\eta,\eta\rab}} \frac{\Gamma(1/2)}{\Gamma(1-1/2\beta)} = \frac{1}{\sqrt{2\lab\eta,\eta\rab}\Gamma(1-1/2\beta)}.
\]
\end{corollary}

In the same way, we can define Donsker's delta in any arbitrary point $a\in\R$:
\begin{proposition}
Let again $\eta\in\cH$. Then
\[
	\delta_a(\lab\cdot,\eta\rab) = \frac{1}{2\pi}\int_\R \exp\lb ix(\lab\cdot,\eta\rab-a)\rb\,\rmd x
\]
exists in $\lb\cN\rb^{-1}_{\mu_\beta}$ as a weak integral in the sense of Theorem \ref{charint} and defines Donsker's delta in $a\in\R$.
\end{proposition}
\begin{proof}
The $T_{\mu_\beta}$-transform of the integrand for $\varphi\in\cN_\C$ is given by:
\[
	\frac{1}{2\pi}\exp(-ixa) \rmE_\beta\lb -\halb x^2\lab\eta,\eta\rab - \halb\lab\varphi,\varphi\rab - x\lab\varphi,\eta\rab \rb .
\]
Hence its absolute value coincides with the $T_{\mu_\beta}$-transform in \eqref{eq:Ttransform} in the case $a=0$. Now we can proceed as in the case $a=0$.
\end{proof}

\section*{Outlook}
\noindent In a forthcoming paper, our aim is to work out the applications of Mittag-Leffler analysis to fractional differential equations. As mentioned in Remark \ref{specialcases} it is possible to choose a nuclear triple to obtain a realization of grey noise, see \cite{Sch90, MM09}. This is done as follows: Choose $\cN$ to be the space of Schwartz test functions $\cS(\R)$ equipped with the scalar product
\[
	(\eta,\varphi)_\alpha := C(\alpha) \int_\R \tilde{\eta}(x) \overline{\tilde{\varphi}(x)} \abs{x}^{1-\alpha}\,\rmd x,\quad\eta,\varphi\in\cS(\R),
\]
for $\alpha\in (0,2)$. Here $\tilde{\eta}$ denotes the Fourier transform of $\eta\in\cS(\R)$. The central Hilbert space $\cH$ is the completion of $\cS(\R)$ with respect to $(\cdot,\cdot)_\alpha$. Then $\ind\in\cH$ and $B_t^{\alpha,\beta} = \lab\cdot,\ind\rab,t\geq 0$ is a generalized grey Brownian motion, which coincides in the case $\beta=1$ with a fractional Brownian motion with Hurst parameter $H=\alpha/2$. This means that Theorem \ref{theorem:Donsker} with $\eta=\ind$ proves the existence of Dirac delta composed with a generalized grey Brownian motion $B^{\alpha,\beta}_t$. Furthermore, there we will show that $\delta(x+B_t^{\beta,\beta})$ is a solution of the time-fractional heat equation with initial value $\delta_x$.

\appendix

\section{Mittag-Leffler function, $M$-Wright function and Fox-H-function}\label{AppendixI}
\noindent This part serves to establish the relations between the Mittag-Leffler function $\rmE_\beta$, the $M$-Wright function $M_\beta$ and Fox-H-function $H$. In fact $\rmE_\beta$ and $M_\beta$ are special cases of $H$ and $\rmE_\beta(-z)$ is the Laplace transform of $M_\beta(z)$.

The H-function was discovered by Charles Fox in 1961 \cite{Fox61} and is a generalization of the G-function of Meijer. The definition is as follows: Let $m,n,p,q\in\N$, $0\leq n\leq p$ and $1\leq m\leq q$. Let $A_i,B_j\in\R$ be positive and $a_i,b_j\in\R$ or $\C$ arbitrary for $1\leq i\leq p,1\leq j\leq q$. Then
\[
	H^{m\,n}_{p\,q}\lb z \left| \begin{matrix} (a_p,A_p) \\ (b_q,B_q) \end{matrix} \right. \rb = H^{m\,n}_{p\,q}\lb z \left| \begin{matrix} (a_1,A_1),\dots,(a_p,A_p) \\ (b_1,B_1),\dots,(b_q,B_q) \end{matrix} \right. \rb = \frac{1}{2\pi i}\int_\cL \Theta(s)z^{-s}\,\rmd s,
\]
where
\[
	\Theta(s) = \frac{\lb \prod_{j=1}^m \Gamma(b_j+sB_j) \rb \lb \prod_{j=1}^n \Gamma(1-a_j-sA_j)\rb}{\lb\prod_{j=m+1}^q \Gamma(1-b_j-sB_j)\rb\lb\prod_{j=n+1}^p \Gamma(a_j+sA_j)\rb}.
\]
For further details concerning the contour $\cL$ and existence of $H$ we refer to \cite{MSH10}. The series expansion of $H$ for $\abs{z}>0$ is given in \cite{PBM90}
\begin{multline}\label{eq:Hseries}
	H^{m\,n}_{p\,q}\lb z \left| \begin{matrix} (a_p,A_p) \\ (b_q,B_q) \end{matrix} \right. \rb = \sum_{i=1}^m \sum_{k=0}^\infty \frac{\prod_{j=1,j\neq i}^m \Gamma(b_j-(b_i+k)\tfrac{B_j}{B_i})}{\prod_{j=m+1}^q \Gamma(1-b_j+(b_i+k)\tfrac{B_j}{B_i})} \\ \times\frac{\prod_{j=1}^n \Gamma(1-a_j+(b_i+k)\tfrac{A_j}{B_i})}{\prod_{j=n+1}^p \Gamma(a_j-(b_i+k)\tfrac{A_j}{B_i})} \frac{(-1)^k z^{(b_i+k)/B_i}}{k!B_i},
\end{multline}
under the condition that $\sum_{j=1}^q B_j - \sum_{j=1}^p A_j >0$ and $B_k(b_j+l) \neq B_j(b_k+s)$ for $1\leq j,k \leq m, j\neq k$ and $l,s\in\N_0$.

The $M$-Wright function $M_\beta$ for $0<\beta<1$ was introduced by Mainardi as an auxiliary function when finding the Green's function to the time-fractional diffusion-wave equation, see e.g.~\cite{Mai96, Mai96b}. Its series expansion is given by
\[
	M_\beta(z) = \sum_{n=0}^\infty \frac{(-z)^n}{n!\Gamma(-\beta n + 1-\beta)},\quad z\in\C.
\]
For more details we refer to \cite{MMP10} and the references therein. We only mention that $M_\beta(t) \geq 0$ for $t\geq 0$.

\begin{lemma}\label{Mbeta=Hfunction}
Due to \eqref{eq:Hseries} the $M$-Wright function $M_\beta$ is a special case of Fox-H-function and for all $z\in\C$ it holds that
\[
	M_\beta(z) = H^{1\,0}_{1\,1}\lb z \left| \begin{matrix} (1-\beta,\beta) \\ (0,1) \end{matrix} \right. \rb.
\]
\end{lemma}


\begin{lemma}\label{intexists}
Let $\rho\geq\halb$. Then the integral
\[
	\int_0^\infty \abs{M_\beta(r/\beta) r^{\rho-1} \exp(-rz)}\,\rmd r
\]
is finite for all $z\in\C$.
\end{lemma}

\begin{proof}
For the proof we use the asymptotic behaviour of the $M$-Wright function given in \cite{MMP10}:
\[
	M_\beta(r/\beta) \sim a(\beta) r^{\tfrac{\beta-1/2}{1-\beta}} \exp\lb-b(\beta)r^{1/(1-\beta)}\rb,\quad r\to\infty,
\]
with $a(\beta)=(2\pi(1-\beta))^{-1/2}$ and $b(\beta)=(1-\beta)/\beta$.
This gives the asymptotic behaviour for the integrand:
\begin{align*}
	&M_\beta(r/\beta)r^{\rho-1}\exp(-r\Re(z)) \\
	&\sim a(\beta) r^{\delta-1} \exp\lb r(-b(\beta)r^{\beta/(1-\beta)}-\Re(z))\rb:=g(r),
\end{align*}
where $\delta=\tfrac{\beta-1/2}{1-\beta}+\rho>0$. This means in particular that there exists $r_1>0$, such that for all $r>r_1$
\[
	\abs{M_\beta(r/\beta)r^{\rho-1}\exp(-r\Re(z)) - g(r)} \leq \abs{g(r)}.
\]
Further we choose $r_2>0$ such that for all $r>r_2$
\[
	-\Re(z) - b(\beta)r^{\beta/(1-\beta)} < 0.
\]
We set $r_0=\max(r_1,r_2)$ and split the integral
\begin{align*}
	&\int_0^\infty M_\beta(r/\beta) r^{\rho-1} \exp\lb -r\Re(z) \rb \rmd r  \\ &= \int_0^{r_0} M_\beta(r/\beta) r^{\rho-1} \exp\lb-r\Re(z)\rb \rmd r + \int_{r_0}^\infty M_\beta(r/\beta) r^{\rho-1} \exp\lb-r\Re(z)\rb \rmd r .
\end{align*}
For the first integral we use the moments of the $M$-Wright function (see \cite{MMP10}):
\[
	\int_0^\infty r^\alpha M_\beta(r) \,\rmd r = \frac{\Gamma(\alpha+1)}{\Gamma(\beta\alpha+1)}
\]
for $\alpha>-1$ and using the coordinate transform $r=s\beta$ we obtain
\begin{align*}
	&\int_0^{r_0} M_\beta(r/\beta) r^{\rho-1} \exp\lb-r\Re(z)\rb \rmd r \leq \exp\lb r_0\abs{\Re(z)} \rb \beta^\rho \int_0^\infty M_\beta(s)s^{\rho-1}\,\rmd s \\
	&= \exp\lb r_0\abs{\Re(z)} \rb \beta^\rho \frac{\Gamma(\rho)}{\Gamma(\beta\rho + 1-\beta)} < \infty.
\end{align*}
For the second integral we can estimate as follows:
\begin{align*}
	&\int_{r_0}^\infty M_\beta(r/\beta) r^{\rho-1} \exp\lb-r\Re(z)\rb \rmd r \\
	&\leq \int_{r_0}^\infty \abs{M_\beta(r/\beta) r^{\rho-1} \exp\lb-r\Re(z)\rb - g(r)} + \abs{g(r)}\,\rmd r \\
	&\leq 2a(\beta)\exp\lb -r_0\Re(z)\rb \int_{r_0}^\infty r^{\delta-1} \exp\lb -r_0b(\beta)r^{\beta/(1-\beta)}\rb \,\rmd r.
\end{align*}
From \cite{Gradshteyn} it is known that for $\Re(\mu)>0$, $\Re(\nu)>0$ and $p>0$
\[
	\int_0^\infty r^{\nu-1} \exp\lb -\mu r^p\rb\,\rmd r = \frac{1}{p} \mu^{-\nu/p} \Gamma(\nu/p).
\]
Thus:
\begin{align*}
	&\int_{r_0}^\infty M_\beta(r/\beta) r^{\rho-1} \exp\lb-r\Re(z)\rb \rmd r \\
	&\leq 2a(\beta) \exp\lb-r_0\Re(z)\rb \frac{1-\beta}{\beta} \lb r_0b(\beta)\rb^{-\delta(1-\beta)/\beta}\Gamma(\tfrac{\delta(1-\beta)}{\beta}) <\infty.
\end{align*}
This finishes the proof.
\end{proof}

\begin{lemma}
The mapping  
\begin{equation}\label{eq:integral}
	\C\ni z\mapsto\int_0^\infty M_\beta(r) r^{\rho-1} \exp(-rz)\,\rmd r\in\C
\end{equation}
is holomorphic for $\rho\geq\halb$.
\end{lemma}

\begin{proof}
Let $z\in\C$, $(z_n)_{n\in\N}$ a sequence in $\C$ converging to $z$. Then for sufficiently large $n\in\N$ the estimate $\abs{z_n}\leq\abs{z}+1$ holds. Thus
\[
	\abs{M_\beta(r) r^{\rho-1} \exp(-rz_n)}\leq M_\beta(r) r^{\rho-1} \exp(r(\abs{z}+1)),\quad r>0.
\]
By Lemma \ref{intexists} the integral
\[
	\int_0^\infty M_\beta(r) r^{\rho-1} \exp(r(\abs{z}+1))\,\rmd r = \int_0^\infty M_\beta(r) r^{\rho-1} \exp(-r(-\abs{z}-1))\,\rmd r
\]
is finite and continuity of \eqref{eq:integral} follows by Lebesgue's dominated convergence. Let now $\gamma$ be any closed and bounded curve in $\C$. Then:
\[
	\int_\gamma \int_0^\infty M_\beta(r) r^{\rho-1} \exp(-rz)\,\rmd r \,\rmd z = \int_0^\infty M_\beta(r) r^{\rho-1} \int_\gamma \exp(-rz)\,\rmd z\,\rmd r = 0,
\]
since the exponential is holomorphic on $\C$. By Morera's theorem \eqref{eq:integral} is holomorphic. 
\end{proof}

\begin{lemma}\label{LaplaceMbeta}
We have for all $z\in\C$ and $\rho\geq 1/2$ that
\begin{equation}\label{eq:LaplaceM}
	\int_0^\infty M_\beta(r) r^{\rho-1} \exp(-rz)\,\rmd r = H^{1\,1}_{1\,2}\lb z \left| \begin{matrix} (1-\rho,1) \\ (0,1),(\beta(1-\rho),\beta) \end{matrix} \right. \rb.
\end{equation}
\end{lemma}

\begin{proof}
Assume first that $\Re(z)>0$. Then Lemma \ref{Mbeta=Hfunction} implies
\[
	\int_0^\infty M_\beta(r) r^{\rho-1} \exp(-rz)\,\rmd r = \int_0^\infty H^{1\,0}_{1\,1}\lb r \left| \begin{matrix} (1-\beta,\beta) \\ (0,1) \end{matrix} \right. \rb r^{\rho-1} \exp(-rz)\,\rmd r.
\]
Using the formula for the Laplace transform of the H-function in \cite{KST06}, we obtain (since $\Re(z)>0$)
\begin{align*}
	\int_0^\infty H^{1\,0}_{1\,1}\lb r \left| \begin{matrix} (1-\beta,\beta) \\ (0,1) \end{matrix} \right. \rb r^{\rho-1} \exp(-rz)\,\rmd r \\= z^{-\rho} H^{1\,1}_{2\,1}\lb z^{-1} \left| \begin{matrix} (1-\rho,1),(1-\beta,\beta) \\ (0,1) \end{matrix} \right. \rb.
\end{align*}
The inversion formula for the H-function, see e.g.~\cite{KST06} yields
\[
	H^{1\,1}_{2\,1}\lb z^{-1} \left| \begin{matrix} (1-\rho,1),(1-\beta,\beta) \\ (0,1) \end{matrix} \right. \rb = H^{1\,1}_{1\,2}\lb z \left| \begin{matrix} (1,1) \\ (\rho,1),(\beta,\beta) \end{matrix} \right. \rb,
\]
and the multiplication rule, see e.g.~\cite{KST06} gives
\[
	z^{-\rho} H^{1\,1}_{1\,2}\lb z \left| \begin{matrix} (1,1) \\ (\rho,1),(\beta,\beta) \end{matrix} \right. \rb = H^{1\,1}_{1\,2}\lb z \left| \begin{matrix} (1-\rho,1) \\ (0,1),(\beta-\rho\beta,\beta) \end{matrix} \right. \rb.
\]
Thus the statement is shown on the halfplane $\Re(z)>0$. Since the  H-function is holomorphic on $\C\backslash\lcb 0\rcb$ and the integral on the left hand side of \eqref{eq:LaplaceM} is holomorphic on $\C$ the equality extends by the identity principle to $\C\backslash\lcb 0\rcb$. Using \eqref{eq:Hseries} we obtain
\[
	H^{1\,1}_{1\,2}\lb z \left| \begin{matrix} (1-\rho,1) \\ (0,1),(\beta-\rho\beta,\beta) \end{matrix} \right. \rb = \sum_{k=0}^\infty \frac{(-z)^k\Gamma(\rho+k)}{k!\Gamma(1-\beta+\beta\rho+\beta k)}.
\]
Hence we can extend the H-function holomorphically to $z=0$ and
\[
	H^{1\,1}_{1\,2}\lb 0 \left| \begin{matrix} (1-\rho,1) \\ (0,1),(\beta-\rho\beta,\beta) \end{matrix} \right. \rb = \frac{\Gamma(\rho)}{\Gamma(1-\beta+\beta\rho)}.
\]
On the other hand, using the Mellin transform of $M_\beta$, see \cite{MMP10}, we find
\[
	\int_0^\infty M_\beta(r) r^{\rho-1} \,\rmd r = \frac{\Gamma(\rho)}{\Gamma(\beta(\rho-1)+1)}.
\]
Thus both sides coincides on $\C$.
\end{proof}

\begin{corollary}\label{Ebeta=LaplaceMbeta}
For all $z\in\C$ the following holds
\[
	\int_0^\infty M_\beta(r) \exp(-rz)\,\rmd r = \rmE_\beta(-z).
\]
\end{corollary}

\begin{proof}
Use Lemma \ref{LaplaceMbeta} with $\rho=1$ and get
\[
	\int_0^\infty M_\beta(r) \exp(-rz)\,\rmd r = \sum_{k=0}^\infty \frac{(-z)^k\Gamma(1+k)}{k!\Gamma(1-\beta+\beta+\beta k)} = \rmE_\beta(-z).
\]
\end{proof}

\section*{Acknowledgement}
\noindent The authors would like to thank FCT - Funda\c{c}\~{a}o para Ci\^{e}ncia e a Tecnologia - for financial support through the project Ref\textsuperscript{\underline{a}} PEst-OE/MAT/UI0219/2011 and PEst-OE/MAT/UI0219/2014. M.~Grothaus, F.~Jahnert and F.~Riemann would like to thank the members of CCM - Centro de Ci\^{e}ncias Matem\'{a}ticas - where this project was initiated, for their warm hospitality. F.~Jahnert and F.~Riemann gratefully acknowledge financial support in the form of a fellowship of the German state Rhineland-Palatinate.

\bibliographystyle{alpha}
\bibliography{bibliography}
 
\end{document}